\documentclass[12pt]{amsart}
\usepackage[T1]{fontenc}
\usepackage{lmodern,amsmath,amssymb,amsthm,booktabs,array}
\usepackage[margin=1in]{geometry}
\usepackage{setspace}
\usepackage[colorlinks=true,linkcolor=blue,citecolor=blue,urlcolor=blue]{hyperref}
\hypersetup{pdftitle={Finite Cyclic Fundamental Groups and Exotic Universal Covers from the Cartwright--Steger Surface},pdfauthor={Anar Akhmedov, B. Doug Park, S\"{u}meyra Sakall{\i}, Sai-Kee Yeung}}
\newtheorem{theorem}{Theorem}[section]
\newtheorem{lemma}[theorem]{Lemma}
\newtheorem{proposition}[theorem]{Proposition}

\theoremstyle{remark}
\newcommand{\CP}{\mathbb{CP}}
\newcommand{\Z}{\mathbb Z}

\newcommand{\im}{\operatorname{im}}

\title[Finite cyclic fundamental groups and exotic universal covers]{Finite Cyclic Fundamental Groups and Exotic Universal Covers from the Cartwright--Steger Surface}

\author{Anar Akhmedov}
\address[A.~Akhmedov]{School of Mathematics, University of Minnesota,
206 Church Street SE, Minneapolis, MN 55455, USA}
\email{akhmedov@umn.edu}
\author{B. Doug Park}
\address[B.~D.~Park]{Department of Pure Mathematics, University of Waterloo,
200 University Avenue W,
Waterloo, ON, N2L 3G1, Canada}
\email{bdpark@uwaterloo.ca}
\author{S\"{u}meyra Sakall{\i}}
\address[S.~Sakall{\i}]{Department of Mathematics and Statistics,
University of South Florida, 4202 E. Fowler Avenue, Tampa, FL 33620, USA}
\email{sumeyras@usf.edu}
\author{Sai-Kee Yeung}
\address[S.-K.~Yeung]{Department of Mathematics, Purdue University,
150 N. University Street, West Lafayette, IN 47907-2067, USA}
\email{yeungs@purdue.edu}

\date{October 13, 2026.}

\begin{document}
\doublespacing
\raggedbottom

\begin{abstract}
We use a particular nodal curve on the Cartwright-Steger surface to construct minimal
symplectic four-manifolds with arbitrary nontrivial finite cyclic fundamental group and fixed topological 
invariants $(e,\sigma,\chi_h,c_1^2)=(43,1,11,89)$.  
Using these symplectic four-manifolds as building blocks, we construct a symplectic 4-manifold that is homeomorphic to $(22m-1)\CP^2\#(21m-1)\overline{\CP}^{\,2}$ 
for any integer $m\geq 2$.
\end{abstract}

\maketitle

\section{Introduction}

The Cartwright--Steger surface is a compact complex ball quotient
$M=\mathbb B^2_{\mathbb C}/\Pi$, where $\mathbb B^2_{\mathbb C}$
is the complex two-ball and $\Pi\subset\operatorname{PU}(2,1)$ is a
torsion-free cocompact lattice. Here a lattice means a discrete subgroup
of the Lie group $\operatorname{PU}(2,1)$ with finite covolume;
cocompactness means that the corresponding quotient is compact.
Since the ball is contractible, $M$ is aspherical and
$\pi_1(M)\cong\Pi$. Its basic topological invariants are
\[
 e(M)=3,\qquad \sigma(M)=1,\qquad H_1(M;\Z)\cong\Z^2
\]
(cf.~\cite{CKY}). We use a nodal totally geodesic curve in $M$\/ to construct
symplectic four-manifolds with controlled fundamental groups.

The constructions of Akhmedov--Sakall{\i}--Yeung \cite{ASY}, and the subgroup
and complement calculations in \cite{AY}, use this curve to obtain simply
connected examples and examples with fundamental groups $\Z/2$ and
$(\Z/2)^2$. Here we introduce a further symplectic sum that permits arbitrary
finite cyclic groups while retaining a Cartwright--Steger building block.

The covering degree of a curve and the order of the final fundamental group
play different roles. Once the curve subgroup is the mod-two kernel in the
Cartwright--Steger lattice, passing to larger covers does not produce larger
cyclic quotients by the same surface-sum argument. We first make this
restriction explicit. We then use a genus-two surface of square $-1$ in the
simply connected fourfold-cover construction. Summing along this surface
with a genus-two surface of square $+1$ in a one-point blowup of $T^4$
produces a new simply connected symplectic manifold with an essential
Lagrangian torus whose complement is simply connected.

\begin{theorem}\label{thm:main}
For every integer $m\ge2$, there is a closed minimal symplectic
four-manifold $Z_m$ satisfying
\[
 \pi_1(Z_m)\cong\Z/m,\qquad
 (e(Z_m),\sigma(Z_m),\chi_h(Z_m),c_1^2(Z_m))=(43,1,11,89).
\]
Its universal cover $\widetilde Z_m$ is symplectic and is homeomorphic,
but not diffeomorphic, to
\[
 (22m-1)\CP^2 \# (21m-1)\overline{\CP}{}^{2},
\]
the connected sum of\/ $22m-1$ copies of the complex projective plane and\/ $21m-1$ copies of the complex projective plane equipped with the noncomplex orientation.  
The deck transformations give a free symplectic $\Z/m$-action on
$\widetilde Z_m$.
\end{theorem}

The cyclic group is introduced by a symplectic torus bundle over a torus.
Its monodromy matrix (see (\ref{eq:monodromy}) below) has cokernel $\Z/m$ after subtracting the identity.
We will make use of a genus-two symplectic sum producing the required
torus complement from the Cartwright--Steger construction. We also prove
a parity lemma (Lemma~\ref{lem:parity}) that applies to any such torus whenever the signature of the
simply connected manifold is nonzero modulo eight. This avoids inferring
the parity of an even-degree cover from the parity of its quotient.

We do not assert that the displayed characteristic numbers are optimal,
or classify the homeomorphism types of the non-simply connected quotients in this paper.  
We plan to deal with these issues in a sequel paper.  
The only finite-index lattice calculation used below is the subgroup
identification established in \cite{AY}, and in particular, no new coset enumeration is required.

Exotic symplectic non-spin four-manifolds with cyclic fundamental group and negative signature were studied in \cite{Torres}.  
More recently, exotic non-spin four-manifolds with positive signature and cyclic fundamental group of order $4m+2$ were constructed in \cite{BSS}.  
Positive-signature non-spin examples with cyclic fundamental group of order $4m$ were constructed in \cite{LT}.

\section{Curves, covers, and the simply connected building block}

Let $M$ be the Cartwright--Steger surface and $\Pi=\pi_1(M)$.
The invariants and curve data used here are
\[
 e(M)=3,\quad \sigma(M)=1,\quad H_1(M;\Z)\cong\Z^2.
\]
Choose the type A nodal curve $D$ used in \cite{ASY,AY}. Its normalization
has genus four, it has one positive ordinary node, and
\begin{equation}\label{eq:D}
 D^2=-1,\qquad K_M\cdot D=9,\qquad p_a(D)=5.
\end{equation}
These numerical data come from \cite{CKY}. Write
\[
 N=\ker\bigl(\Pi\longrightarrow H_1(M;\Z/2)\bigr).
\]
The subgroup computation of \cite[Proposition 6.3]{AY} identifies
\begin{equation}\label{eq:subgroup}
 \im\bigl(\pi_1(D)\longrightarrow\Pi\bigr)=N,
 \qquad \Pi/N\cong(\Z/2)^2.
\end{equation}
Here $\pi_1(D)$ means the fundamental group of the singular curve, not
of its normalization. Smoothing the node preserves this ambient image:
the new handle contains a branch-joining loop, and its dual vanishing
cycle is nullhomotopic in the ambient coordinate ball. This is the local
calculation in \cite[Lemma 6.8]{AY}.

\subsection*{Numerical behavior in covers}
For an unramified holomorphic cover $p:V\to M$ of degree $d$, the
full pullback of the nodal curve satisfies
\[
 (p^*D)^2=-d,\qquad K_V\cdot p^*D=9d,\qquad
 p_a(p^*D)=4d+1.
\]
When the full pullback is connected, smoothing its positive nodes gives
an embedded symplectic surface of genus $4d+1$ in the same homology class.
The smoothing genus equals the arithmetic genus $p_a(p^*D)$.
For a divisor $C$ on a smooth complex surface $V$, we use
$p_a(C)=1+(C^2+K_V\cdot C)/2$. It differs from the genus of the
normalization, which removes the nodes without joining their branches.
If the pullback of the normalization has components of degrees $d_j$, then
\[
 \sum_j d_j=d,\qquad g_j=3d_j+1,
\]
where $g_j$ is the genus of the $j$-th component.  
Connectedness of the singular pullback does not imply connectedness
of its normalization.

Let $M_4\to M$ be the cover corresponding to $N$. Equation
\eqref{eq:subgroup} implies that the inverse image of $D$ consists of four
disjoint copies, each mapping with degree one. For a further cover
$V_n\to M_4$ of degree $n$, the full pullback of a chosen copy is connected.
After smoothing, it has genus $4n+1$ and square $-n$. A regular blowup at
a point of the smoothing gives square $-(n+1)$ and an exceptional
geometric dual. Blowing up $n$ regular points instead gives square $-2n$.
The exceptional spheres meet the surface once each. This order of
operations supplies the meridian relation used in \cite{AY}.

We now specify the surface sum used to compare the possible quotients.
Let $H\le\Pi$ have finite index, and let $M_H\to M$ be its connected
cover. Choose a connected lift of the smoothed curve and blow up a
regular point of it. Denote the resulting manifold by
$\widehat M_H=M_H\#\overline{\CP}^{\,2}$ and the proper transform
of the chosen curve by $F_H$. Suppose that a simply connected
symplectic four-manifold $X$ contains a symplectic surface $F_X$ of the
same genus as $F_H$ and of opposite self-intersection. For an admissible
normal-bundle gluing, form the symplectic sum (cf.~\cite{Gompf})
\[
 Z_H=\widehat M_H\#_{F_H=F_X}X.
\]
The following proposition applies whenever this matching positive block
exists.

\begin{proposition}\label{prop:obstruction}
The image of $\pi_1(F_H)$ in $\pi_1(\widehat M_H)\cong H$ is
$H\cap N$, up to conjugacy, and the symplectic sum just defined satisfies
\[
 \pi_1(Z_H)\cong H/(H\cap N)
 \cong\im\bigl(H\to(\Z/2)^2\bigr).
\]
In particular, this symplectic sum construction gives only the trivial
group, $\Z/2$, or $(\Z/2)^2$.
\end{proposition}
\begin{proof}
The smoothed surface still has image $N$ in $\Pi$. The restriction of its
cover to the chosen component has subgroup given by the inverse image of
$H$ under its fundamental-group homomorphism. Its image in $H$ is
$H\cap N$. Since $N$ is normal in $\Pi$, this subgroup is normal in $H$.
The exceptional geometric dual kills the meridian of $F_H$. 
The fundamental group of the complement of $F_X$ 
is normally generated by the meridian of $F_X$, because $X$\/ is simply connected. 
The gluing kills this group
and imposes precisely the surface subgroup on $H$. Van Kampen gives
the displayed quotient, and the final isomorphism is the first
isomorphism theorem.
\end{proof}

We use the fourfold cover, so that this quotient is trivial.

\begin{proposition}\label{prop:seed}
There is a simply connected minimal symplectic four-manifold $Y$ with
\[
 (e(Y),\sigma(Y),\chi_h(Y),c_1^2(Y))=(38,2,10,82)
\]
containing a symplectic genus-two surface $R$ of square $-1$.
\end{proposition}
\begin{proof}
Smooth one lifted nodal curve in $M_4$, and blow up a regular point.
The proper transform $F_-$ in $M_4\#\overline{\CP}^{\,2}$ has genus
five and square $-2$; its exceptional sphere $E$ meets it once.
Blowing up a point does not change the fundamental group, so
\[
 \pi_1(M_4\#\overline{\CP}^{\,2})\cong\pi_1(M_4)\cong N.
\]
Under these identifications, the image of $\pi_1(F_-)$ is all of $N$.

The minimal simply connected symplectic block $X_{2,3}$ of
\cite{ASY} has $e=9$, $\sigma=-1$ and square-zero symplectic surfaces
$F_2,F_3$ of genera two and three, intersecting positively once.
Their symplectic resolution $F_+$ has genus five and square $+2$.
A sufficiently small parallel symplectic copy $P$ of $F_2$ has
\begin{equation}\label{eq:P}
 g(P)=2,\qquad P^2=0,\qquad P\cdot F_+=1.
\end{equation}
The geometric intersection statement, including the choice of the
parallel copy, is verified in \cite[Lemma 4.1]{AY}.

Form the symplectic sum
\[
 Y=(M_4\#\overline{\CP}^{\,2})\#_{F_-=F_+}X_{2,3}.
\]
Choose an admissible gluing matching the marked intersections with $E$
and $P$ and their local normal framings. 
The meridian of $F_-$ is trivial, since the punctured $E$ is a bounding disk. 
The fundamental group of the complement of $F_+$ dies with its meridian, 
and the loops on $F_+$ kill $N$ on the
$F_-$ side. Hence $\pi_1(Y)=1$. The symplectic sum formulas give
\[
 e(Y)=13+9+16=38,\qquad \sigma(Y)=3-1=2.
\]

Glue the punctured surfaces $E\setminus\operatorname{int}D^2$ and
$P\setminus\operatorname{int}D^2$ along their boundary circles, and
call the resulting embedded surface $R\subset Y$. Here the symplectic property requires compatible local data, in
addition to the oriented framing calculation. By a local symplectic
isotopy, arrange that both dual surfaces meet their summing surfaces
symplectically orthogonally at the marked points. Choose matching
Darboux charts and take the base gluing to be the identity on the marked
disks and the normal-fiber phase to be constant, and rescale the forms
on the summands to match the areas of the summing surfaces. In the local symplectic sum model
$xy=\lambda$, the base coordinate is $z$ and the two dual disks are
$z=0$ on the respective sides. Their punctured parts join through the
annulus $z=0$, $xy=\lambda$. This annulus is complex for the standard
local compatible structure and is therefore symplectic. It agrees with
the dual surfaces near the ends of the neck, so the standard sum form
makes the whole joined surface symplectic (cf.~\cite{Gompf}).
Thus $R$\/ is symplectic,
has genus two, and its normal Euler number is
\[
 R^2=E^2+P^2=-1.
\]
The constant-phase matching of the local normal framings makes the
relative Euler numbers additive, as in \cite[Lemma 5.2]{AY}.

Finally, $M_4$ is aspherical, and hence $\pi_2(M_4)=0$.
Let $b:M_4\#\overline{\CP}^{\,2}\to M_4$ be the blowdown map,
and let $S$ be any embedded exceptional sphere in the blowup.
The composite map $S^2\cong S\hookrightarrow
M_4\#\overline{\CP}^{\,2}\xrightarrow{b}M_4$ is nullhomotopic,
so $b_*[S]=0$. The kernel of $b_*$ on integral second homology is
$\Z[E]$, the subgroup of integer multiples of the exceptional class.
Thus $[S]=k[E]$ for an integer $k$. Since $[S]^2=[E]^2=-1$,
we have $k=\pm1$. Consequently $[S]\cdot[F_-]=\pm1$,
and no exceptional sphere is disjoint from $F_-$. The other summand is minimal.
Neither summand is an $S^2$-bundle with the summing surface as a section.
Usher's theorem in \cite{Usher} proves that $Y$ is minimal.
\end{proof}

\section{A genus-two sum with simply connected torus complement}

The next sum produces the torus needed to prescribe the cyclic group.
We write
\[
 T^4=T^2_{ab}\times T^2_{cd}
\]
with its product symplectic form and circle coordinates $a,b,c,d$.
Let $A=T^2_{ab}\times\{(0,0)\}$ and
$B=\{(0,0)\}\times T^2_{cd}$.
Resolve their single positive intersection to obtain a genus-two
symplectic surface $\Sigma_0\subset T^4$ with $\Sigma_0^2=2$.
Blow up a regular point of $\Sigma_0$, and let $\Sigma$ denote its
proper transform in
\[
 Q=T^4\#\overline{\CP}^{\,2},\qquad \Sigma\subset Q.
\]
Then $g(\Sigma)=2$, $\Sigma^2=1$, and the exceptional sphere meets
$\Sigma$ once.

Choose $b_0,d_0\ne0$ and set
\[
 T=\{b=b_0,\ d=d_0\}\cong S^1_a\times S^1_c.
\]
It is a Lagrangian torus disjoint from $A\cup B$. Choose the local
chart used for the resolution, and the ball used for the blowup, with
closures disjoint from $T$. These operations then leave a neighborhood
of $T$ unchanged, so $T$ is disjoint from $\Sigma$ and remains
Lagrangian in $Q$. Its normal bundle is trivial.

The statement below holds for any pair $(Y,R)$ with the indicated
properties. We subsequently apply it to the pair constructed in
Proposition~\ref{prop:seed}.

\begin{theorem}\label{thm:bridge}
Let $Y$ be a simply connected minimal symplectic four-manifold containing
a symplectic genus-two surface $R$ of square $-1$. Form
\[
 L=Y\#_{R=\Sigma}\, Q.
\]
Then $L$ is minimal and simply connected. The torus $T$ survives as an
essential Lagrangian torus, and
\[
 \pi_1(L\setminus\operatorname{int}\nu T)=1.
\]
Moreover,
\[
 e(L)=e(Y)+5,\qquad \sigma(L)=\sigma(Y)-1,
 \qquad \chi_h(L)=\chi_h(Y)+1,\qquad c_1^2(L)=c_1^2(Y)+7.
\]
\end{theorem}
\begin{proof}
Choose the symplectic sum away from $T$. The complement of this product
torus has a particularly simple form:
\[
 T^4\setminus\operatorname{int}\nu T
 \cong T^2_{ac}\times(T^2_{bd}\setminus\operatorname{int}D^2).
\]
Consequently,
\begin{equation}\label{eq:productgroup}
 \pi_1(Q\setminus\operatorname{int}\nu T)
 \cong\langle a\rangle\times\langle c\rangle\times F(b,d).
\end{equation}
Here $\langle a\rangle$ and $\langle c\rangle$ are infinite cyclic
groups, so their product is $\Z^2$, and $F(b,d)$ is the free group
on $b,d$. We use the commutator convention
$[x,y]=xyx^{-1}y^{-1}$. The meridian of $T$ is therefore
$[b,d]=bdb^{-1}d^{-1}$, up to orientation. Blowing up at a
point away from $T$ does not change this group.

Put $C_Q=Q\setminus\operatorname{int}(\nu\Sigma\cup\nu T)$ and
$C_Y=Y\setminus\operatorname{int}\nu R$. The exceptional sphere of
$Q$ punctured at its unique intersection with $\Sigma$ gives a disk in
$C_Q$ bounded by the meridian $\mu_\Sigma$. Thus $\mu_\Sigma=1$
in $\pi_1(C_Q)$. Filling back $\Sigma$ kills only the normal closure
of this meridian, so
\begin{equation}\label{eq:capgroup}
 \pi_1(C_Q)\cong\pi_1(Q\setminus\operatorname{int}\nu T).
\end{equation}
Choose the deleted disk in the $b,d$ torus disjoint from the coordinate
wedge $\{d=0\}\cup\{b=0\}$. The punctured torus retracts onto this
wedge, whose loops give the free generators $b,d$. The coordinate tori
$A,B$ carry the four loops $a,b,c,d$; resolving their intersection only
changes their based representatives inside a ball disjoint from $T$.
Choose the blowup point away from these loops. Their images therefore
normally generate the full group in \eqref{eq:productgroup}, including
its nonabelian punctured-torus factor. The boundaries of the summing
surfaces are circle bundles of Euler numbers $+1$ and $-1$; the loops
used in van Kampen are lifts of these surface generators, together with
the circle fiber.

Because $Y$ is simply connected, $\pi_1(C_Y)$ is normally generated by
the meridian $\mu_R$. The gluing identifies $\mu_R$ with the inverse
of the trivial $\mu_\Sigma$, so the entire $Y$-side group dies.
The remaining boundary relations kill the four surface loops on the
$Q$ side, and hence $a,b,c,d$. Equations
\eqref{eq:productgroup} and \eqref{eq:capgroup} now give
$\pi_1(L\setminus\operatorname{int}\nu T)=1$.
This proof does not assume that $C_Y$ itself is simply connected.
Restoring $\nu T$ also gives $\pi_1(L)=1$.

Let $C=L\setminus\operatorname{int}\nu T$. 
The long exact homology sequence of the pair $(L,C)$, excision and K\"unneth theorems give
\[
 H_2(L;\Z)\xrightarrow{x\mapsto x\cdot[T]}\Z
 \longrightarrow H_1(C;\Z)=0.
\]
Thus the homology class $[T]$ is primitive and nonzero. The local product form makes $T$
Lagrangian before the sum, and the sum preserves the symplectic form in
a neighborhood of $T$.

The summand $Y$ is minimal. Every exceptional sphere in $Q$ has class
$\pm E_Q$, by the same asphericity argument applied to $T^4$,
and intersects $\Sigma$. If a summand were an $S^2$-bundle with the
genus-two summing surface as a section, its base would have genus two
and its fundamental group would be the genus-two surface group.
This is incompatible with $\pi_1(Y)=1$ and with $\pi_1(Q)=\Z^4$.
Usher's theorem therefore gives minimality.

Since $e(Q)=1$, $\sigma(Q)=-1$, and $e(\Sigma)=-2$, the sum formulas
give $e(L)=e(Y)+1+4$ and $\sigma(L)=\sigma(Y)-1$.
The formulas for $\chi_h=(e+\sigma)/4$ and $c_1^2=2e+3\sigma$ follow.
\end{proof}

Applying Theorem \ref{thm:bridge} to Proposition \ref{prop:seed} gives
\begin{equation}\label{eq:L}
 (e(L),\sigma(L),\chi_h(L),c_1^2(L))=(43,1,11,89),
 \qquad b_2^+(L)=21,\quad b_2^-(L)=20.
\end{equation}
The minimal simply connected symplectic manifold $L$ is irreducible
by \cite{HK}, in the usual convention allowing homotopy-sphere summands.

\section{Introducing cyclic groups and controlling parity}
We introduce the group by monodromy, using a torus bundle with a
square-zero symplectic section. 

\begin{lemma}\label{lem:transfer}
Let $L$ be a simply connected irreducible minimal symplectic four-manifold containing
an essential Lagrangian torus $T$ with simply connected exterior
$C=L\setminus\operatorname{int}\nu T$. For every $m\ge2$ there is a
minimal symplectic manifold $Z_m$ with $\pi_1(Z_m)\cong\Z/m$ and
with the same Euler characteristic and signature as $L$. The construction
retains $C$ as a submanifold.
\end{lemma}
\begin{proof}
Consider the matrix
\begin{equation}\label{eq:monodromy}
 A_m=\begin{pmatrix}m+2&-1\\1&0\end{pmatrix}\in\mathrm{SL}(2,\Z).
\end{equation}
Choose an orientation-preserving diffeomorphism $\phi_m$ of the torus $T^2$,
inducing $A_m$, which is the identity on a neighborhood of a small closed
disk $D_0\subset T^2$.  Such a
representative exists because the mapping class group of the torus
with a disk fixed maps onto $\mathrm{SL}(2,\Z)$. We may choose it to be 
area preserving: start with a fixed area form $\omega$; the difference
$\phi_m^*\omega-\omega$ vanishes near $D_0$ and has integral zero.
Relative Moser isotopy on the complement of $D_0$ corrects the map
without changing its mapping class or its restriction near $D_0$.

Let
\[
 E_m=\bigl(T^2\times[0,1]/(x,1)\sim(\phi_m(x),0)\bigr)\times S^1.
\]
Write $t$ for the mapping-torus coordinate and $s$ for the last circle coordinate.
Projection to $(t,s)$ makes $E_m$ a torus bundle
$\rho_m:E_m\to S^1_t\times S^1_s$. Since $\phi_m$ preserves
$\omega$, the form
$\omega+dt\wedge ds$ descends to a symplectic form on $E_m$.
For $x_0\in\operatorname{int}D_0$, the torus
$B=\{x_0\}\times S^1_t\times S^1_s$ is the image of a section
of $\rho_m$. Its tangent plane is the base plane, on which
$\omega+dt\wedge ds$ restricts to $dt\wedge ds$; thus $B$ is
symplectic. It has trivial normal bundle. Its product neighborhood comes from the disk
on which $\phi_m$ is the identity.

Let $u,v$ generate the fundamental group of the punctured fiber and
let $\psi_m$ be the automorphism induced by $\phi_m$, using a basepoint
on the boundary of the fixed disk. Since the map is the identity there,
$\psi_m$ fixes the boundary commutator $[u,v]$. With $t$ positively
oriented in the mapping-torus coordinate, the section
exterior has presentation
\[
 \pi_1(E_m\setminus\operatorname{int}\nu B)
 =\left\langle u,v,t,s\ \middle|\
 t^{-1}ut=\psi_m(u),\quad t^{-1}vt=\psi_m(v),\quad
 [s,u]=[s,v]=[s,t]=1\right\rangle.
\]
The section-boundary longitudes are $t,s$, and its meridian is
$[u,v]$, up to orientation. Perturb the form on $L$ to make the
essential Lagrangian torus $T$ symplectic, and match its area with
that of $B$ by rescaling the form on $E_m$ \cite[Lemma 1.6]{Gompf}. Form
\[
 Z_m=L\#_{T=B}E_m.
\]
The simply connected $L$ exterior kills the entire boundary image.
Consequently van Kampen imposes $t=s=1$ and $[u,v]=1$. After
abelianization of the fundamental group of the punctured fiber, the remaining monodromy relations are
precisely $(A_m-I)\Z^2$. Thus
\[
 \pi_1(Z_m)\cong\Z^2/(A_m-I)\Z^2\cong\Z/m.
\]
Indeed,
\[
 A_m-I=\begin{pmatrix}m+1&-1\\1&-1\end{pmatrix},
 \qquad \det(A_m-I)=-m.
\]
The cyclic quotient can also be read directly from the columns of
$A_m-I$: in additive notation the generators $\bar u,\bar v$
satisfy
\[
 (m+1)\bar u+\bar v=0,\qquad -\bar u-\bar v=0.
\]
The second relation eliminates $\bar v=-\bar u$, and the first then
becomes $m\bar u=0$. Thus the quotient is cyclic of order $m$.

The torus bundle $E_m$ is aspherical, by the homotopy exact sequence
of its torus fibration over $T^2$. It therefore has no exceptional
spheres. It has $e(E_m)=0$ and $\sigma(E_m)=0$: Euler characteristic
is multiplicative for the bundle, and the product with $S^1$ admits
an orientation-reversing diffeomorphism. Neither summand is a ruled
surface with the summing torus as a section: $L$ is simply connected,
whereas a ruled surface with a torus section has fundamental group
$\Z^2$; $E_m$ is aspherical. The perturbed form on $L$ is still minimal because smooth
irreducibility excludes exceptional spheres. Usher's theorem therefore
gives minimality of the sum. A torus sum adds neither Euler characteristic nor signature in
this case. Finally, the entire exterior $C$ is retained, and no
modification affects its interior.
\end{proof}

Non-spinness of $L$ alone would not determine the parity of an even-degree
cover of $Z_m$. The following lemma provides a class in the retained exterior.

\begin{lemma}\label{lem:parity}
Let $L$ be a closed simply connected smooth oriented four-manifold, and let
$T\subset L$ be a square-zero torus with simply connected exterior $C$.
If $\sigma(L)\not\equiv0\pmod8$, then $C$ contains a closed embedded
oriented surface $S$ of odd self-intersection.
\end{lemma}

\begin{proof}
Write $\Lambda=H_2(L;\Z)$ with its unimodular intersection pairing,
$t=[T]$, and $f(x)=x\cdot t$. The exact sequence
\[
 H_2(C;\Z)\longrightarrow\Lambda\xrightarrow{f}\Z
 \longrightarrow H_1(C;\Z)=0
\]
shows that $f$ is surjective and $t$ is primitive. Let $K=\ker f$.
Suppose, for contradiction, that every vector in $K$ has even square. The function
\[
 q:\Lambda/2\Lambda\longrightarrow\mathbb{Z}/2,\qquad q(\bar x)=x^2\pmod2
\]
is linear. The mod-two image of $K$ is exactly $\ker\bar f$: when
$f(y)$ is even, choose $u$ with $f(u)=1$ and replace $y$ by
$y-f(y)u$, which has the same mod-two class and lies in $K$.
Thus $q=\lambda\bar f$ for $\lambda\in\mathbb{Z}/2$.

If $\lambda=0$, the lattice is even, so its signature is divisible by
eight. If $\lambda=1$, then $t$ is characteristic, and van der Blij's lemma
in \cite{MH} gives $t^2\equiv\sigma(L)\pmod8$. Both cases contradict
$t^2=0$ and the signature hypothesis. Therefore $K$ contains an element 
of odd square. Exactness lifts it to $H_2(C;\Z)$. Represent this class
by an immersed closed oriented surface in the interior, resolve its
double points, and join components by tubes if necessary. This gives
an embedded surface $S\subset\operatorname{int}C$ representing the
same class, with odd square.
\end{proof}

\begin{proof}[Proof of Theorem \ref{thm:main}]
Construct $L$ as above, and apply Lemma \ref{lem:transfer}. Equation
\eqref{eq:L} gives the invariants of $Z_m$. Lemma \ref{lem:parity}
applies because $\sigma(L)=1$. The odd-square surface $S$ is retained
in $Z_m$. Its fundamental group has trivial image in $\pi_1(Z_m)$,
since its inclusion factors through the simply connected exterior $C$.
Let $p:\widetilde Z_m\to Z_m$ be the universal covering map.
Its degree is $m$, since $\pi_1(Z_m)\cong\Z/m$. The image of
$\pi_1(S)$ is trivial, so its monodromy action on the $m$ points
of a covering fiber is trivial. Consequently the restricted covering
$p^{-1}(S)\to S$ is the disjoint union of $m$ degree-one coverings:
\[
 p^{-1}(S)=S_1\sqcup\cdots\sqcup S_m,\qquad
 p|_{S_j}:S_j\xrightarrow{\cong}S.
\]
The normal bundle of each $S_j$ is the degree-one pullback of the
normal bundle of $S$. Their normal Euler numbers therefore agree,
so $[S_j]^2=[S]^2$ is odd. An even intersection form would make
every integral homology class have even square; the class $[S_j]$
excludes this. Thus the intersection form of $\widetilde Z_m$ is
odd for every $m$, including even $m$.

The symplectic form pulls back to $\widetilde Z_m$, and
\[
 e(\widetilde Z_m)=43m,\qquad \sigma(\widetilde Z_m)=m.
\]
Since the cover is simply connected,
\[
 b_2(\widetilde Z_m)=43m-2,\qquad
 b_2^+(\widetilde Z_m)=22m-1,\qquad
 b_2^-(\widetilde Z_m)=21m-1.
\]
Its intersection form is odd, indefinite and unimodular, hence
\[
 (22m-1)\langle1\rangle\oplus(21m-1)\langle-1\rangle.
\]
Freedman's theorem in \cite{Freedman}, together with the vanishing
Kirby--Siebenmann invariant of the smooth cover, gives the stated
homeomorphism. Taubes's theorem in \cite{Taubes} gives a nonzero
Seiberg--Witten invariant because $b_2^+(\widetilde Z_m)>1$.
The standard connected sum has vanishing Seiberg--Witten invariants:
it splits into two summands with positive $b_2^+$. The cover is
therefore not diffeomorphic to the standard model. The deck group
preserves the pulled-back symplectic form and acts freely.
\end{proof}

\section{Geography and integral homology}

The construction separates the choice of group order from the
characteristic numbers of the quotient. For every $m\ge2$,
\[
 b_1(Z_m)=0,\qquad b_2^+(Z_m)=21,\qquad b_2^-(Z_m)=20,
\]
where the Betti numbers are taken over $\mathbb Q$. Integral duality gives
\[
 H_1(Z_m;\Z)\cong\Z/m,\qquad
 H_2(Z_m;\Z)\cong\Z^{41}\oplus\Z/m,\qquad H_3(Z_m;\Z)=0.
\]
Indeed, abelianization determines $H_1$; the free rank of $H_2$ follows
from $e=43$; and Poincar\'e duality with the universal coefficient theorem
identifies its torsion with $\operatorname{Ext}(H_1,\Z)$.
For a summary of the topological invariants, we have the following table.  

\bigskip 

\begin{center}
\begin{tabular}{@{}lrrrr@{}}
\toprule
Manifold & $e$ & $\sigma$ & $\chi_h$ & $c_1^2$\\
\midrule
Cartwright--Steger seed $Y$ & $38$ & $2$ & $10$ & $82$\\
Genus-two sum $L$ & $43$ & $1$ & $11$ & $89$\\
Cyclic quotient $Z_m$ & $43$ & $1$ & $11$ & $89$\\
Universal cover $\widetilde Z_m$ & $43m$ & $m$ & $11m$ & $89m$\\
\bottomrule
\end{tabular}
\end{center}

\bigskip

For $m=2$, the universal cover is an exotic
$43\CP^2\#41\overline{\CP}{}^{2}$; for $m=3$, it is an exotic
$65\CP^2\#62\overline{\CP}{}^{2}$. 
These examples have free
symplectic cyclic actions with quotients of positive signature one.
For comparison, exotic $k\CP^2\#(k-2)\overline{\CP}{}^{2}$ with signature $2$ were constructed for all odd $k\geq 19$ in \cite{ASY}.   
Exotic $l\CP^2\#(l-3)\overline{\CP}{}^{2}$ with signature $3$ were constructed for all odd 
$l\geq 29$ in \cite{AS}.  
We refer to \cite{AHP}, \cite{AP-2008} and \cite{AP-2021} for the existence of exotic smooth structures for other signatures.  
It is not known to the authors whether the exotic smooth structures that we constructed in this paper are in fact diffeomorphic to the exotic smooth structures in \cite{ASY}, \cite{AS} and \cite{AHP}.   
Finally, our construction differs from the intermediate-double-cover sum in
\cite{AY}: the extra genus-two sum changes the quotient invariants and
supplies the torus needed to prescribe the cyclic order.

\section*{Acknowledgments}

Magma was used to assist with fundamental group computations, and an LLM-based tool was used solely for grammar checking. All mathematical ideas are due to the authors, who reviewed the material and take full responsibility for the content of the manuscript.


\begin{thebibliography}{99}
	
\bibitem{AHP}  
A.~Akhmedov, M.~C.~Hughes and B.~D.~Park,
\textit{Geography of simply connected nonspin symplectic\/ $4$-manifolds with positive signature}, 
Pacific J.~Math. \textbf{261} (2013), no.~2, 257--282.	

\bibitem{AP-2008}  
A.~Akhmedov and B.~D.~Park, 
\textit{New symplectic\/ $4$-manifolds with nonnegative signature},
J. G\"{o}kova Geom. Topol. GGT {\bf 2} (2008), 1--13.

\bibitem{AP-2021}
A.~Akhmedov and B.~D.~Park, 
\textit{Geography of simply connected nonspin symplectic\/ $4$-manifolds with positive signature, II}, 
Canad. Math. Bull. \textbf{64} (2021), no.~2, 418--428.

\bibitem{AS}
A.~Akhmedov, S.~Sakall{\i}, 
\textit{On the geography of simply connected nonspin symplectic\/ $4$-manifolds with nonnegative signature}, 
Topology Appl. \textbf{206} (2016), 24--45. 

\bibitem{ASY}
A.~Akhmedov, S.~Sakall{\i}, and S.-K.~Yeung,
\textit{Complex ball quotients and new symplectic\/ $4$-manifolds with
nonnegative signatures}, 
Taiwanese J.~Math. \textbf{28} (2024), no.~1,
29--53. 

\bibitem{AY}
A.~Akhmedov and S.-K.~Yeung,
\textit{Torus surgery constraints and intermediate covers of the
Cartwright-Steger surface}, 
arXiv:2610.06442. 
Subgroup computation and supporting files at:
\url{https://github.com/akhmedov-coder/cartwright-steger-torus-surgery}.

\bibitem{BSS}
R.~\.I.~Baykur and A.~I.~Stipsicz and Z.~Szab\'o,
\textit{Smooth structures on four-manifolds with finite cyclic fundamental groups},
arXiv:2406.09007.

\bibitem{CKY}
D.~I.~Cartwright, V.~Koziarz, and S.-K.~Yeung,
\textit{On the Cartwright--Steger surface}, J.~Algebraic Geom.
\textbf{26} (2017), no.~4, 655--689.

\bibitem{Freedman}
M.~H.~Freedman, 
\textit{The topology of four-dimensional manifolds},
J.~Differential Geom. \textbf{17} (1982), no.~3, 357--453.

\bibitem{Gompf}
R.~E.~Gompf, 
\textit{A new construction of symplectic manifolds},
Ann.~of Math. (2) \textbf{142} (1995), no.~3, 527--595.

\bibitem{HK}
M.~J.~D.~Hamilton and D.~Kotschick,
\textit{Minimality and irreducibility of symplectic four-manifolds},
Int.~Math.~Res.~Not. (2006), Art.~ID 35032.

\bibitem{LT}
R.~Ladu and S.~Tagliente, 
\textit{On smooth structures over\/ $4$-manifolds with fundamental group of even order}, 
arXiv:2603.29794.

\bibitem{MH}
J.~Milnor and D.~Husemoller, 
\textit{Symmetric bilinear forms},
Ergebnisse der Mathematik und ihrer Grenzgebiete, vol.~73,
Springer-Verlag, Berlin, 1973.

\bibitem{Taubes}
C.~H.~Taubes, 
\textit{The Seiberg--Witten invariants and symplectic forms},
Math.~Res.~Lett. \textbf{1} (1994), no.~6, 809--822.

\bibitem{Torres}
R. Torres, 
\textit{Geography and botany of irreducible non-spin symplectic $4$-manifolds with abelian fundamental group}, 
Glasg.~Math.~J. \textbf{56} (2014), no.~2, 261--281.


\bibitem{Usher}
M.~Usher, 
\textit{Minimality and symplectic sums}, Int.~Math.~Res.~Not.
(2006), Art.~ID 49857.

\end{thebibliography}
\end{document}